\documentclass[11pt, fleqn]{amsart}

\usepackage[square,compress,comma, numbers,sort]{natbib}
\usepackage[colorlinks=true, citecolor=red, linkcolor=blue]{hyperref}
\usepackage{amsfonts,mathtools}
\usepackage{amsmath}

\allowdisplaybreaks[4]

\usepackage{amssymb}
\usepackage{color}
\usepackage{float}
\newcommand{\abs}[1]{\left\lvert #1 \right\rvert}

\DeclarePairedDelimiterXPP\pk[1]{\mathbb{P}}\{ \}{}{ #1}
\DeclarePairedDelimiterXPP\E[1]{\mathbb{E}}\{ \}{}{	#1}

\usepackage{xparse}% http://ctan.org/pkg/xparse

\NewDocumentCommand{\ceil}{s O{} m}{%
	\IfBooleanTF{#1} % starred
	{$\left\lceil#3\right\rceil$} % \ceil*[..]{..}
	{#2\lceil#3#2\rceil} % \ceil[..]{..}
}
\NewDocumentCommand{\floor}{s O{} m}{%
	\IfBooleanTF{#1} % starred
	{$\left\lfloor#3\right\rfloor$}
	{#2\lfloor#3#2\rfloor}
}

\def\one{\mbox{1\hspace{-4.25pt}\fontsize{12}{14.4}\selectfont\textrm{1}}} % 11pt    
\definecolor{c20}{rgb}{0.,0.7,0.}
\definecolor{c30}{rgb}{0.,0.,1.}
\definecolor{c40}{rgb}{1,0.1,0.7}
\definecolor{c50}{rgb}{1,0,0}
\definecolor{c60}{rgb}{1,0.9,0.1}
\definecolor{c70}{rgb}{0.50,1.00,0.00}

\numberwithin{equation}{section}
\newtheorem{theo}{Theorem}[section]
\newtheorem{sat}[theo]{Proposition}
\newtheorem{de}[theo]{Definition}
\newtheorem{lem}[theo]{Lemma}

\newtheorem{example}[theo]{Example}
\newtheorem{korr}[theo]{Corollary}
\newtheorem{remark}[theo]{Remark}

\numberwithin{equation}{section}
\newtheorem{theorem}{Theorem}[section]

\newtheorem{lemma}{Lemma}[section]

\newcommand{\COM}[1]{}

\newcommand{\BQN}{\begin{eqnarray}}
\newcommand{\EQN}{\end{eqnarray}}
\newcommand{\BQNY}{\begin{eqnarray*}}
	\newcommand{\EQNY}{\end{eqnarray*}}

\newcommand{\BS}{\begin{sat}}
	\newcommand{\ES}{\end{sat}}
\newcommand{\BT}{\begin{theo}}
	\newcommand{\ET}{\end{theo}}
\newcommand{\BK}{\begin{korr}}
	\newcommand{\EK}{\end{korr}}

\newcommand{\BEX}{\begin{example}}
	\newcommand{\EEX}{\end{example}}

\newcommand{\BD}{\begin{de}}
	\newcommand{\ED}{\end{de}}

\newcommand{\BIT}{\begin{itemize}}
	\newcommand{\EIT}{\end{itemize}}
\newcommand{\BDI}{\begin{description}}
	\newcommand{\EDI}{\end{description}}

\newcommand{\BRM}{\begin{remark}}
	\newcommand{\ERM}{\end{remark}}

\newcommand{\BEL}{\begin{lem}}
	\newcommand{\EEL}{\end{lem}}

\title{Proportional reinsurance for fractional Brownian risk model}

\author{Krzysztof K\c{e}pczy\'nski}
\address{Krzysztof K\c{e}pczy\'nski,	 Mathematical Institute, University of Wroc\l aw, pl. Grunwaldzki 2/4, 50-384 Wroc\l aw, Poland}
\email{Krzysztof.Kepczynski@math.uni.wroc.pl}

\begin{document}
\bigskip

\date{\today}
 \maketitle
 {\bf Abstract:} This paper investigates ruin probabilities for a two-dimensional fractional Brownian risk model with a proportional reinsurance scheme. We focus on joint and simultaneous ruin probabilities in a finite-time horizon. The risk processes of both insurance and reinsurance companies are composed of a large number of i.i.d. sub-risk processes, representing independent businesses. We derive the asymptotics as the initial capital tends to infinity.\\
 {\bf Key Words:} Fractional Brownian motion; asymptotics; ruin probability; two-dimensional risk model; proportional reinsurance\\
 {\bf AMS Classification:} Primary 60G15; secondary 60G70
%\usepackage{natbib}
%\usepackage{graphicx}

%\maketitle
\section{Introduction}
Consider a two-dimensional risk model with a proportional reinsurance scheme. Suppose that two companies: insurance and reinsurance, share claims in proportions $\sigma_1, \sigma_2>0,$ where $\sigma_1+\sigma_2=1,$ and receive premiums at rates $c_1, c_2>0,$ respectively. Let $R_i$ denote the risk process of $i$-th company
\BQNY
R_i(t) := a_i + c_it - \sigma_i X(t), t \geq 0,
\EQNY
where $X(t)$ describes the accumulated claims up to time $t,$ $a_i>0$ is the initial capital and $c_i>0$ is the premium rate, $i=1,2.$  

In the literature various processes of accumulated claims are investigated, with particular emphasis to both L\'evy and Gaussian processes. The study of Gaussian processes in risk theory was initiated in the fundamental work of Iglehart \cite{iglehart}, where $X(t)$ is a standard Brownian motion and appears as the limit in the so called diffusion approximation regime. 
In an important work by Michna \cite{michna} it was argued that the class of {\it fractional Brownian motions} can serve as a right approximation of the accumulated claims process.

The modern risk theory focuses on the ruin probability in multi-dimensional risk models. The exact distribution of ruin probabilities are known only in a few specific cases in dimension two: Brownian motion \cite{kepczynski} and spectrally one-sided L\'evy processes \cite{palmowski1, palmowski2, michna2}. It motivates to study the asymptotic properties, bounds and Laplace transform of the ruin probability.

Having introduced the risk processes, we can distinguish the following ruin types:

$\Diamond$ \textit{Simultaneous ruin} occurs, when exists $t\in[0,T]$ such that both companies are ruined at time $t$
\BQNY
    \pi_{sim}^{T}(a_1,a_2) = \mathbb{P}\left(\inf_{ t \in [0,T]} \left(R_1(t), R_2(t)\right)<0\right).
\EQNY
$\Diamond$ \textit{Joint ruin} occurs, when both companies are ruined in time interval $[0,T],$ not necessarily at the same moment
\BQNY
    \pi_{and}^{T}(a_1,a_2) = \mathbb{P}\left(\inf_{ s \in [0,T]}R_1(s)<0 \text{ and }\inf_{ t \in [0,T]} R_2(t)<0\right).
\EQNY

$\Diamond$ \textit{'At least one' ruin} occurs, when at least one insurance company is ruined in time interval $[0,T]$
\BQNY%
    \pi_{or}^{T}(a_1,a_2) = \mathbb{P}\left(\inf_{  s \in [0,T]} R_1(s)<0 \text{ or }\inf_{  t \in [0,T]} R_2(t)<0\right).
\EQNY 
We refer to \cite{palmowski1, palmowski2, debicki5, debicki1, krystecki, debicki2, debicki,  foss, Ji, kepczynski, michna2} for relevant recent discussions about two-dimensional risk models. For example, models with Gaussian claim processes are considered in \cite{debicki5, debicki1, krystecki, debicki2, debicki, Ji, kepczynski} while L\'evy claim processes are investigated in \cite{palmowski1, palmowski2, debicki2, foss, michna2}. The above papers consider mainly asymptotics, and, in L\'evy claims case, also contain exact distributions and Laplace transforms of ruin probabilities.

In this contribution we study the two-dimensional fractional Brownian risk model, i.e. suppose that $X(t)$ is a fractional Brownian motion $B_H(t)$, that is, a centered Gausssian process with stationary increments, covariance function $r(s,t) = \frac{1}{2}\left(\abs{t}^{2H} + \abs{s}^{2H} - \abs{t-s}^{2H}\right)$ and $B_H(0)=0$ a.s. We focus on joint and simultaneous ruin probabilities in the case that the risk processes of both insurance and reinsurance companies are composed of a large number of i.i.d. sub-risk processes $R_i^{(k)},$ representing independent businesses. That is, we investigate
\BQN\label{2.dim.sim} 
\pi_{sim}^T(N) :=\mathbb{P}\left( \exists t \in [0,T]: \sum_{k=1}^{N}R^{(k)}_1(t) < 0, \sum_{k=1}^{N} R^{(k)}_2(t)<0\right)
\EQN
and 
\BQN\label{2.dim.and}
\pi_{and}^T(N) :=\mathbb{P}\left( \sup\limits_{s \in [0,T]} \sum_{k=1}^{N}R^{(k)}_1(s) < 0, \sup\limits_{t \in [0,T]} \sum_{k=1}^{N} R^{(k)}_2(t)<0\right),
\EQN
where $R_i^{(k)}(t) = a_i + c_i t - \sigma_i B_H^{(k)}(t), k =1,\dots, N.$ 
We concentrate on the asymptotic behavior of ruin probabilities (\ref{2.dim.sim}) and (\ref{2.dim.and}), as $N \to \infty.$ In Theorem \ref{sim}, which contains the main contribution of this paper, we find exact asymptotics of the simultaneous ruin probability. In Lemma \ref{fbm_1dim} and Theorem \ref{fbm_2dim} we study asymptotics of the joint ruin probability: logarithmic and exact. 

Let us briefly mention the following standard notation for two given positive functions $f(\cdot)$ and $g(\cdot)$ which we use in this contribution. We write $f(x) =g(x)(1 +o(1))$ if $\lim\limits_{x \to \infty}f(x)/g(x) = 1$ and $f(x) =o(g(x))$ if $\lim\limits_{x \to \infty} f(x)/g(x) = 0.$ We also write $1-\Phi\left(x\right) =  \Psi\left(x\right) := \mathbb{P}\left(\mathcal{N} > x \right),$ where $\mathcal{N}$ is a standard normal random variable. 

The remainder of the paper is organized as follows. In Section \ref{main.results} we formalize the problem and present main results of this contribution. Section \ref{proofs} contains auxiliary facts and proofs.
\section{Main results}\label{main.results}
We begin with several observations and assumptions.\\
$\Diamond$ Using that, for $\sigma_1, \sigma_2>0$
\BQNY
\pi_{sim}^{T}(a_1,a_2) &=& \mathbb{P}\left(\inf_{ t \in [0,T]} \left(a_1 + c_1t - \sigma_1 X(t), a_2 + c_it - \sigma_2 X(t)\right)<0\right) \\
&=& \mathbb{P}\left(\inf_{ t \in [0,T]} \left(\frac{a_1}{\sigma_1} + \frac{c_1}{\sigma_1}t - X(t),  \frac{a_2}{\sigma_2} + \frac{c_2}{\sigma_2}t - X(t)\right)<0\right)
\EQNY
and 
\BQNY
\pi_{and }^{T}(a_1,a_2) = \mathbb{P}\left(\inf_{ s \in [0,T]} \left(\frac{a_1}{\sigma_1} + \frac{c_1}{\sigma_1}s - X(s)\right)<0 \text{ and }\inf_{ t \in [0,T]} \left(\frac{a_2}{\sigma_2} + \frac{c_2}{\sigma_2}t - X(t)\right)<0\right)
\EQNY
without loss of generality we shall suppose that $\sigma_i=1, i=1,2.$

$\Diamond$ Note that $\sum_{k=1}^{N} B_H^{(k)}(t) =_{d} \sqrt{N} B_H(t),$ 
where $=_{d}$ denotes equality in distribution. Thus we can rewrite the ruin probabilities (\ref{2.dim.sim}) and (\ref{2.dim.and}) as 
\BQNY
\pi_{sim}^T(N) = \mathbb{P}\left( \exists t \in [0,T]: \left( B_H(t) - c_1 \sqrt{N} t\right) > a_1\sqrt{N}, \left( B_H(t) - c_2 \sqrt{N} t\right) > a_2\sqrt{N}\right),
\EQNY
and 
\BQNY
\pi_{and}^{T}(N) = \mathbb{P}\left( \sup\limits_{t \in [0,T]} \left( B_H(t) - c_1 \sqrt{N} t\right) > a_1\sqrt{N}, \sup\limits_{t \in [0,T]} \left( B_H(t) - c_2 \sqrt{N} t\right) > a_2\sqrt{N}\right).
\EQNY
$\Diamond$ Due to the symmetry of the two-dimensional problem, in the rest of the paper without loss of generality we assume that $c_1>c_2>0$.\\
$\Diamond$ We note that if $a_1 \geq a_2>0,$ then the lines $a_1+c_1t$ and $a_2+c_2t$ do not intersect over $[0,\infty)$ and our problem degenerates to the one-dimensional ruin, i.e.  $$\pi_{sim}^T(N) = \pi_{and}^T(N) = \mathbb{P} \left(\sup\limits_{t\in[0,T]} \left( B_H(t) - c_1 \sqrt{N} t\right) > a_1\sqrt{N} \right),$$ which was considered in \cite{sm_buf_rev}. 
To avoid dimension-reduction, we shall assume $0<a_1<a_2.$

Let  $t^{*} := \frac{a_2-a_1}{c_1-c_2}$ denotes the unique point of intersection of the above lines. In the rest of the work we focus on the case $t^{*}<T.$\\
$\Diamond$ It follows from general theory on extremes of Gaussian processes that in the one-dimensional case the point that maximizes variance of $\frac{B_H(t)}{a_i+c_it}$ corresponds to the logarithmic asymptotics;  see, e.g., \cite{piterbarg}. That is $$\lim_{N\to\infty} \frac{\log\mathbb{P}\left(\sup\limits_{t \geq 0} \left(B_H(t) - c_i\sqrt{N} t \right) > a_i\sqrt{N}\right)}{N} = 
-\frac{1}{2} \left[ \sup_{t \geq 0}\mathbb{V}ar\left(\frac{B_H(t)}{a_i+c_it}\right) \right]^{-1}.$$
Elementary calculations show that $$t_i := \arg \sup_{t \geq 0}\mathbb{V}ar\left(\frac{B_H(t)}{a_i+c_it}\right) =  \frac{a_i H}{c_i (1-H)}, \text{ for } i=1,2.$$
It turns out that points $t_1$ and $t_2$ play important role also in two-dimensional case.  As we show later, the order between $t_1,$ $t_2$ and $t^{*}$ affects the asymptotics of $\pi^{T}_{sim}(N)$ and $\pi^{T}_{and}(N),$ as $N \to \infty.$

Let introduce some constants that play a crucial role in our main results. First we define the classical \textit{Pickands constant} $$\mathcal{H}_{2H} = \lim\limits_{T \to \infty} \frac{1}{T}\mathbb{E}\left( \exp\left( \sup\limits_{t \in [0,T]} \left( \sqrt{2}B_H(t)-t^{2H} \right)\right) \right).$$
It is known that $\mathcal{H}_{2H} \in (0, \infty)$ for $H \in (0, 1]$ and $\mathcal{H}_1 = 1, \mathcal{H}_2 = 1/\sqrt{\pi};$ see, e.g., \cite{sm_buf_rev, piterbarg, Ji}.

Furthermore, for any continuous function $d(\cdot)$ such that $d(0)=0,$ define $$\tilde{\mathcal{H}}^{d}_1=\lim\limits_{t \to \infty} \mathbb{E}\left( \exp\left( \sup\limits_{t \in [-T,T]} \left(\sqrt{2} B_{1/2}(t)-|t| -d(t)\right) \right) \right)$$
whenever the limit exists; see, e.g., \cite{Ji}.

\subsection{Simultaneous ruin}\label{simultaneous}
This section contains exact asymptotics of the simultaneous ruin probability. Ji \& Robert \cite{Ji} considered a similar problem in the infinite-time horizon. We use a similar argument to extend Theorem $3.1$ in \cite{Ji} to the case $T \in (0,\infty).$

First we recall asymptotics of 
$$\psi^T(N;a,c) := \mathbb{P}\left( \sup\limits_{ t \in [0,T]} \left( B_H(t) - c_1 \sqrt{N} t\right) > a_1\sqrt{N}\right),$$
which will play an important role in further analysis. 

The proof of the following lemma can be found in \cite{sm_buf_rev}.
\begin{lemma}\label{fbm_1dim}
Suppose that $H \in [0,1]$ and $a,c, T>0.$ Let $m:=m(a,c,H) = \left(\frac{a}{1-H}\right)^{1-H}\left(\frac{c}{H}\right)^H.$\\
$(i)$ If $T>\frac{H}{1-H}\frac{a}{c},$ then, as $N \to \infty,$
    $$\psi^T(N;a,c) = \frac{\mathcal{H}_{2H}}{\pi} \frac{1}{\sqrt{H(1-H)}} \left(\frac{m}{\sqrt{2}}\right)^{\frac{1}{H}-1}N^{\frac{H-1}{2}}\frac{1}{\sqrt{2\pi}\sqrt{N}}\frac{1}{m} e^{-\frac{m^2}{2}N}(1+o(1)).$$
$(ii)$ If $T=\frac{H}{1-H}\frac{a}{c},$ then, as $N \to \infty,$
    $$\psi^T(N;a,c) = \frac{\mathcal{H}_{2H}}{2\pi} \frac{1}{\sqrt{H(1-H)}} \left(\frac{m}{\sqrt{2}}\right)^{\frac{1}{H}-1}N^{\frac{H-1}{2}}\frac{1}{\sqrt{2\pi}\sqrt{N}}\frac{1}{m} e^{-\frac{m^2}{2}N}(1+o(1)).$$
$(iii,a)$ If $T<\frac{H}{1-H}\frac{a}{c}$ and $H \in (0,\frac{1}{2})$ then, as $N \to \infty,$ 
    $$\psi^T(N;a,c) = \mathcal{H}_{2H} \frac{T^{2H-1}(a+cT)^{\frac{1}{H}-1}}{cT-H(a+cT)}\frac{N^{\frac{H-2}{2}}}{2^{\frac{H}{2}}} \frac{1}{\sqrt{2\pi}\sqrt{N}} \frac{T^H}{a+cT} e^{-\frac{\left(a+cT\right)^2}{2T^{2H}}N}(1+o(1)).$$
$(iii,b)$ If $T<\frac{H}{1-H}\frac{a}{c}$ and $H =\frac{1}{2}$ then, as $N \to \infty,$ 
    $$\psi^T(N;a,c) =  \frac{1}{\sqrt{2\pi}\sqrt{N}}\frac{2 a\sqrt{T}}{(a - c T) (a + c T)} e^{-\frac{\left(a+cT\right)^2}{2T}N}(1+o(1)).$$
$(iii,c)$ If $T<\frac{H}{1-H}\frac{a}{c}$ and $H \in (\frac{1}{2},1)$ then, $N \to \infty,$
     $$\psi^T(N;a,c) =  \frac{1}{\sqrt{2\pi}\sqrt{N}}\frac{T^{2H}}{(a + c T)} e^{-\frac{\left(a+cT\right)^2}{2T^{2H}}N}(1+o(1)).$$
\end{lemma}
Since in several cases asymptotics of two-dimensional ruin probabilities reduces to one-dimensional one, for the sake of brevity we give the main result of this section in a language of one-dimensional ruin probability given in Lemma \ref{fbm_1dim}. Let denote $$A_i = \frac{|(a_i+c_it^{*})H - c_it^{*}|}{(a_i+c_it^{*})t^{*}}, i=1,2.$$
\begin{theorem}\label{sim} Suppose that $t^{*}<T.$\\
$(i)$ If $t^{*}<t_1,$ then, as $N \to \infty,$ $$\pi^{T}_{sim}(N) = \psi^T(N;a_1,c_1) (1+o(1)).$$
$(ii)$ If $t^{*}=t_1,$ then, as $N \to \infty,$ $$\pi^{T}_{sim}(N) = \frac{1}{2}\psi^T(N;a_1,c_1)(1+o(1)).$$
$(iii,a)$ If $t_1<t^{*}<t_2$ and $H<1/2,$ then, as $N \to \infty,$ $$\pi^{T}_{sim}(N) = \frac{A_1+A_2}{2^{1/(2H)}t^{*} A_1 A_2}\mathcal{H}_{2H}\left(\frac{a_1+c_1t^{*}}{(t^{*})^{H}}\sqrt{N}\right)^{1/H-2} \Psi\left(\frac{a_1+c_1t^{*}}{(t^{*})^{H}}\sqrt{N}\right) (1+o(1)).$$
$(iii,b)$ If $t_1<t^{*}<t_2$ and $H=1/2,$ then, as $N \to \infty,$ $$\pi^{T}_{sim}(N)= \tilde{\mathcal{H}}^{d}_1 \Psi\left(\frac{a_1+c_1t^{*}}{(t^{*})^{H}}\sqrt{N}\right) (1+o(1)),$$
with $d(t) = 2 t^{*}A_2 |t| \one\{t<0\}+2 t^{*}A_1 |t| \one\{t \geq 0\}.$\\
$(iii,c)$ If $t_1<t^{*}<t_2$ and $H>1/2,$ then, as $N \to \infty,$ $$\pi^{T}_{sim}(N) = \Psi\left(\frac{a_1+c_1t^{*}}{(t^{*})^{H}}\sqrt{N}\right)(1+o(1)).$$
$(iv)$ If $t_1<t_2=t^{*},$ then, as $N \to \infty,$ $$\pi^{T}_{sim}(N) = \frac{1}{2}\psi^T(N;a_2,c_2)(1+o(1)).$$
$(v)$ If $t_1<t_2<t^{*},$ then, as $N \to \infty,$ $$\pi^{T}_{sim}(N)= \psi^T(N;a_2,c_2)(1+o(1)).$$
\end{theorem}
\begin{remark}
Ji \& Robert \cite{Ji} showed that for function $d(t) = 2 t^{*}A_2 |t| \one\{t<0\}+2 t^{*}A_1 |t| \one\{t \geq 0\}$ the constant $\tilde{\mathcal{H}}_1^d$ is well-defined, positive and finite.

\end{remark}
\subsection{Joint ruin}\label{component.wise}
This section contains logarithmic and exact asymptotics of the joint ruin probability. K\c{e}pczy\'nski \cite{kepczynski} and Lieshout \& Madjes \cite{mandjes} considered related problems for a standard Brownian motion in finite-time and infinite-time horizons, respectively.

The following lemma gives logarithmic asymptotics of the joint ruin probability.
\begin{lemma}\label{proposition} Let $H \in (0,1]$. Then, as $N \to \infty,$
$$\lim_{N \to \infty} \frac{\log\left(\pi^T_{and}(N)\right)}{N} = - \frac{1}{2} \inf_{0\leq s,t \leq T} \frac{1}{\max\left(\sigma_1^2(s),  \sigma_2^2(t)\right)}\left(1+\frac{(c(s,t)-r(s,t))^2}{1-r^2(s,t)}\one\{r(s,t)<c(s,t)\}\right),$$
where $r(s,t) := \mathbb{C}orr\left(\frac{B_H(s)}{a_1+c_1s}, \frac{B_H(t)}{a_2+c_2t}\right)= \frac{1}{2s^Ht^H}\left(t^{2H}+s^{2H}-|t-s|^{2H}\right),$ $c(s,t) := \max\left(\frac{\sigma_2(t)}{\sigma_1(s)}, \frac{\sigma_1(s)}{\sigma_2(t)}\right)$ and $\sigma_i^2(t) := \mathbb{V}ar\left(\frac{B_H(t)}{a_i+c_it}\right), i=1,2.$
\end{lemma}

In the following proposition we consider the special case $H = 1.$
\begin{sat}\label{H1}Suppose that $H=1.$\\
$(i)$ If $t^*\leq T,$ then $$\pi^T_{and}(N) = \Psi\left(\frac{a_2+c_2T}{T}\sqrt{N}\right).$$
$(ii)$ If $t^*> T,$ then $$\pi^T_{and}(N) = \Psi\left(\frac{a_1+c_1T}{T}\sqrt{N}\right).$$
\end{sat}
\begin{remark}\label{corollary2}Suppose that $H=1.$\\
$(i)$ If $t^*\leq T,$ then, as $N \to \infty,$ $$\pi^T_{and}(N) = \frac{1}{\sqrt{2\pi}\sqrt{N}}\frac{T}{a_2+c_2T} e^{-\frac{(a_2+c_2T)^2}{2T^2}N}(1+o(1)).$$
$(ii)$ If $t^*> T,$ then, as $N \to \infty,$ $$\pi^T_{and}(N) = \frac{1}{\sqrt{2\pi}\sqrt{N}}\frac{T}{a_1+c_1T} e^{-\frac{(a_1+c_1T)^2}{2T^2}N}(1+o(1)).$$
\end{remark}
\begin{theorem}\label{fbm_2dim} Suppose that $t^{*}<T.$\\
$(i)$ If $t^*<t_1$ then, as $N \to \infty,$  $$\pi^T_{and}(N) = \psi^T(N;a_1,c_1)(1+o(1)).$$
$(ii)$ If $t_1<t_2<t^*<T,$ then, as $N \to \infty,$
$$\pi^T_{and}(N)=\psi^T(N;a_2,c_2)(1+o(1)).$$
\end{theorem}
\begin{remark}
Theorem \ref{fbm_2dim} gives an exact asymptotics, but only in cases which lead to dimension-reduction scenario. In the remainder case $t_1\leq t^* \leq t_2$ the analysis of $\pi_{and}^{T}(N)$ goes out of the approach presented in this contribution and we can get only logarithmic asymptotics as in Lemma \ref{proposition}. We refer to a recent contribution by K\c{e}pczy\'nski \cite{kepczynski} where case $t_1\leq t^{*} \leq t_2$ was solved for the special case Brownian motion, i.e. $H=1/2.$
\end{remark}
\section{Proofs}\label{proofs}
\begin{proof}[Proof of Theorem \ref{sim}]
We divide the proof on the following three cases: $t^*<t_1,$ $t_1 \leq t^* \leq t_2$ and $t_1<t_2<t^*.$

Case $(i): t^*<t_1.$ We have
\BQNY
    \pi^T_{sim}(N)&\geq& \mathbb{P}\left(\sup_{t \in [t^*, T]}\left( B_H(t) - c_1 \sqrt{N} t\right) > a_1 \sqrt{N}\right)\\
    &\geq& \mathbb{P}\left(\sup_{t \in [0, T]} \left(  B_H(t) - c_1 \sqrt{N} t\right) > a_1 \sqrt{N}\right) 
    - \mathbb{P}\left(\sup_{t \in [0,t^*]} \left( B_H(t) - c_1 \sqrt{N} t\right) > a_1 \sqrt{N}\right)
\EQNY
and
\BQNY
    \pi^T_{sim}(N)&\leq \mathbb{P}\left(\sup\limits_{t \in [0,T]} \left( B_H(t) - c_1 \sqrt{N}t \right) > a_1 \sqrt{N} \right).
\EQNY
From Lemma \ref{fbm_1dim} we obtain, as $N \to \infty,$ 
$$\mathbb{P}\left(\sup_{t \in [0,t^*]} \left( B_H(t) - c_1 \sqrt{N} t\right) > a_1 \sqrt{N}\right) = o\left( \mathbb{P}\left(\sup_{t \in [0, T]} \left(  B_H(t) - c_1 \sqrt{N} t\right) > a_1 \sqrt{N}\right) \right).$$
Thus, as $N \to \infty,$ $$\pi^T_{sim}(N)= \mathbb{P}\left(\sup_{t \in [0,T]} \left(B_H(t) - c_1 \sqrt{N} t\right) > a_1 \sqrt{N}\right)(1+o(1)).$$

Case $(ii-iv): t_1 \leq t^* \leq t_2.$ Let 
$$Z(t) = \frac{B_H(t)}{g(t)}, \text{ where } g(t) = \max\left( a_1+c_1t,  a_2+c_2t\right)$$  
and 
$$\sigma^2_Z(t) = \mathbb{V}ar(Z(t)) = \min\left( \mathbb{V}ar \left( \frac{B_H(t)}{a_1+c_1t}\right), \mathbb{V}ar \left( \frac{B_H(t)}{a_2+c_2t}\right) \right)= \frac{t^{2H}}{g^2(t)}, t \geq 0.$$
We have $$\pi^T_{sim}(N) = \mathbb{P}\left( \sup\limits_{t \in [0,T] } Z(t) > \sqrt{N}\right).$$
Elementary calculations show that 
$$t_{opt} = \arg\sup_{t \geq 0} \sigma_Z^2(t)$$ is the unique point that maximizes $\sigma^2_Z(t)$ over $[0, \infty).$ 

We have a lower bound
\BQNY
    \pi^T_{sim}(N)&\geq& \mathbb{P}\left(\sup_{t \geq 0}Z(t)>\sqrt{N}\right)
    - \mathbb{P}\left(\sup_{t \geq T} Z(t) >  \sqrt{N}\right)
\EQNY
and an upper bound
\BQNY
    \pi^T_{sim}(N)&\leq \mathbb{P}\left(\sup\limits_{t \geq 0} Z(t)> \sqrt{N} \right).
\EQNY
Since $\lim_{t \to \infty} Z(t) = 0$ a.s., the process $\{Z(t): t \geq 0 \}$ has bounded sample paths. Hence from Borell-TIS inequality (see Theorem $2.6.1$ in \cite{alder}) we obtain that for all sufficiently large $N,$ where $C_0 = \mathbb{E} \left(\sup\limits_{t \geq T } Z(t)\right)< \infty$ it holds
$$\mathbb{P}\left( \sup\limits_{t \geq T } Z(t) > \sqrt{N}\right) \leq \exp\left(-\frac{1}{2 \sup\limits_{t \geq T } \sigma_Z^2(t)} \left(\sqrt{N} - C_0\right)^2\right).$$
Note that $t_{opt} = t^* \in [0,T)$ is the unique maximum point of $\sigma^2_Z(t), t \geq 0$ and $\sup\limits_{t \geq T } \sigma_Z^2(t) < \sigma^2_Z(t_{opt}).$

Hence, we obtain, as $N\to\infty,$ $$\mathbb{P}\left( \sup\limits_{t \geq T } Z(t) > \sqrt{N}\right) = o \left(\mathbb{P}\left( \sup\limits_{t \geq 0 } Z(t) > \sqrt{N}\right)\right).$$

Thus, as $N \to \infty,$ 
\BQNY
    \lefteqn{\pi^T_{sim}(N) = \mathbb{P}\left(\sup_{t \geq 0} Z(t )> \sqrt{N}\right)(1+o(1)) =}\\
    &=& \mathbb{P}\left( \exists t \geq 0: \left( B_H(t) - c_1 t\right) > a_1N^{\frac{1}{2(1-H)}}, \left( B_H(t) - c_2 t\right) > a_2N^{\frac{1}{2(1-H)}}\right)(1+o(1))
\EQNY
and the thesis follows from Theorem $3.1$ in \cite{Ji}.

Case $(v): t_1<t_2<t^*.$ We have 
$$\pi^T_{sim}(N) \geq \mathbb{P}\left(\sup_{t \in [0,t^*]} \left(B_H(t) - c_2 \sqrt{N} t\right) > a_2 \sqrt{N}\right)$$
and 
$$\pi^T_{sim}(N) \leq \mathbb{P}\left(\sup_{t \in [0,T]} \left(B_H(t) - c_2 \sqrt{N}t\right) > a_2 \sqrt{N} \right).$$
From Lemma \ref{fbm_1dim} we obtain, as $N \to \infty,$ 
$$\mathbb{P}\left(\sup_{t \in [0,t^*]} \left(B_H(t) - c_2 \sqrt{N} t\right) > a_2 \sqrt{N}\right) = \mathbb{P}\left(\sup_{t \in [0,T]} \left(B_H(t) - c_2 \sqrt{N}t\right) > a_2 \sqrt{N} \right) \left(1+o(1)\right).$$
Thus, as $N \to \infty,$ $$\pi^T_{sim}(N) = \mathbb{P}\left(\sup_{t \in [0,t^*]} \left(B_H(t) - c_2 \sqrt{N} t\right) > a_2 \sqrt{N}\right)(1+o(1)).$$
This completes the proof.
\end{proof}
The following lemma gives logarithmic asymptotics of a joint survival function for supremum of two centered and bounded Gaussian processes. Its proof can be found in \cite{debicki} (Remark $5$).
\begin{lemma}\label{tandem}
Let $\{X_1(s): s \in \mathcal{T}_1\}$ and $\{X_2(t): t \in \mathcal{T}_2\}$ be two centered and bounded $\mathbb{R}$-valued Gaussian processes. Then, for $q_1, q_2>0,$ as $N \to \infty,$
\BQNY
\lefteqn{
\frac{\log\mathbb{P}\left(\sup\limits_{s \in \mathcal{T}_1}X_1(s) > q_1 \sqrt{N}, \sup\limits_{t \in \mathcal{T}_2}X_2(t) > q_2 \sqrt{N} \right)}{N}=} \\
&=&
-\frac{1}{2} \inf_{s \in \mathcal{T}_1, t \in \mathcal{T}_2} \frac{1}{\max\left(\sigma_1(s)/q_1, \sigma_2(t)/q_2\right)^2} \left(1+\frac{(c_q(s,t)-r(s,t))^2}{1-r^2(s,t)}1_{\{r(s,t) < c_{q}(s,t)\}}\right)(1+o(1)),
\EQNY
where $\sigma_i(t) = \sqrt{\mathbb{V}ar(X_i(t))},$  $r(s,t) = \mathbb{C}orr\left(X_1(s), X_2(t)\right)$ and
$c_q(s,t) = \max\left(\frac{q_1}{\sigma_1(s)}\frac{\sigma_2(t)}{q_2}, \frac{\sigma_1(s)}{q_1}\frac{q_2}{\sigma_2(t)}\right).$
\end{lemma}
\begin{proof}[Proof of Lemma \ref{proposition}] 
It is sufficient to observe that $\left\{\frac{B_H(t)}{a_i+c_it}: t \geq 0\right\}$ is centered and bounded $\mathbb{R}$-valued Gaussian processes, for $i=1,2.$
We note that $$\sigma_i(t) = \frac{t^{H}}{a_i+c_it}\text{ and } \mathcal{T}_i = [0,T], c(s,t) = \max\left(\frac{\sigma_2(t)}{\sigma_1(s)}, \frac{\sigma_1(s)}{\sigma_2(t)}\right)$$ and $$r(s,t) = \mathbb{C}orr\left(\frac{B_H(s)}{a_1+c_1s}, \frac{B_H(t)}{a_2+c_2t}\right) = \frac{1}{2s^Ht^H}\left(t^{2H} + s^{2H} - |t-s|^{2H}\right).$$
Hence Lemma \ref{tandem} implies the thesis.
\end{proof}
\begin{proof}[Proof of Proposition \ref{H1}]
We have that $$\pi^T_{and}(1) =  \mathbb{P}\left(\sup_{t \in [0,T]} \left(\mathcal{N} t- c_1 t\right) > a_1 , \sup_{t \in [0,T]} \left(\mathcal{N}t - c_2 t\right) > a_2\right), \text{ where } \mathcal{N}\sim \mathcal{N}(0,1).$$
Observe that, for $i=1,2,$ we have $\left \{\sup\limits_{t\in[0,T]}(\mathcal{N}t -c_it)>a_i\right \} = \left \{ \mathcal{N}T -c_iT>a_i \right \}.$

Hence
\BQNY
    \left \{\sup_{t\in[0,T]}(\mathcal{N}t -c_1t)>a_1, \sup_{t\in[0,T]}(\mathcal{N}t-c_2t)>a_2\right \}  = \left \{\mathcal{N} > \max\left(\frac{a_1+c_1T}{T}, \frac{a_2+c_2T}{T}\right) \right \}.
\EQNY
Finally, we obtain that
$$\pi^T_{and}(1) = \mathbb{P}\left(\mathcal{N} > \max\left(\frac{a_1+c_1T}{T}, \frac{a_2+c_2T}{T}\right)\right).$$
\end{proof}
\begin{proof}[Proof of Remark \ref{corollary2}]
The proof follows straightforwardly from Proposition \ref{H1} and the fact that, as $x \to \infty,$ $$\mathbb{P}\left(\mathcal{N} > x\right) = \frac{1}{\sqrt{2 \pi}} e^{-\frac{x^2}{2}}.$$
\end{proof}
\begin{proof}[Proof of Theorem \ref{fbm_2dim}]
We divide the proof on two scenarios: $t^*<t_1$ and $t_1<t_2<t^*.$\\
Case $(i): t^*<t_1.$ We have
\BQNY
    \pi^T_{and}(N)
    &\geq& \mathbb{P}\left(\sup_{t \in [0, T]} \left(  B_H(t) - c_1 \sqrt{N} t\right) > a_1 \sqrt{N}\right) 
    - \mathbb{P}\left(\sup_{t \in [0,t^*]} \left( B_H(t) - c_1 \sqrt{N} t\right) > a_1 \sqrt{N}\right)
\EQNY
and
\BQNY
    \pi^T_{and}(N)&\leq \mathbb{P}\left(\sup\limits_{t \in [0,T]} \left( B_H(t) - c_1 \sqrt{N}t \right) > a_1 \sqrt{N} \right).
\EQNY
From Lemma \ref{fbm_1dim} we obtain, as $N \to \infty,$ 
$$\mathbb{P}\left(\sup_{t \in [0,t^*]} \left( B_H(t) - c_1 \sqrt{N} t\right) > a_1 \sqrt{N}\right) = o\left( \mathbb{P}\left(\sup_{t \in [0, T]} \left(  B_H(t) - c_1 \sqrt{N} t\right) > a_1 \sqrt{N}\right) \right).$$
Thus, as $N \to \infty,$ $$\pi^T_{and}(N)= \mathbb{P}\left(\sup_{t \in [0,T]} \left(B_H(t) - c_1 \sqrt{N} t\right) > a_1 \sqrt{N}\right)(1+o(1)).$$
Case $(ii): t_1<t_2<t^*.$ We have 
$$\pi^T_{and}(N) \geq \mathbb{P}\left(\sup_{t \in [0,t^*]} \left(B_H(t) - c_2 \sqrt{N} t\right) > a_2 \sqrt{N}\right)$$
and 
$$\pi^T_{and}(N) \leq \mathbb{P}\left(\sup_{t \in [0,T]} \left(B_H(t) - c_2 \sqrt{N}t\right) > a_2 \sqrt{N} \right).$$
From Lemma \ref{fbm_1dim} we obtain, as $N \to \infty,$ 
$$\mathbb{P}\left(\sup_{t \in [0,t^*]} \left(B_H(t) - c_2 \sqrt{N} t\right) > a_2 \sqrt{N}\right) = \mathbb{P}\left(\sup_{t \in [0,T]} \left(B_H(t) - c_2 \sqrt{N}t\right) > a_2 \sqrt{N} \right) \left(1+o(1)\right).$$
Thus, as $N \to \infty,$ $$\pi^T_{and}(N) = \mathbb{P}\left(\sup_{t \in [0,t^*]} \left(B_H(t) - c_2 \sqrt{N} t\right) > a_2 \sqrt{N}\right)(1+o(1)).$$
This completes the proof.
\end{proof}
\section*{Acknowledgments}
K. K\c{e}pczy\'nski was partially supported by NCN Grant No 2018/31/B/ST1/00370 (2019-2022).

\end{document}